\documentclass{article}

\usepackage{microtype}
\usepackage{graphicx}
\usepackage{subfigure}
\usepackage{booktabs} 

\usepackage{hyperref}

\usepackage{amsmath}
\usepackage{amssymb}
\usepackage{mathtools}
\usepackage{amsthm}

\usepackage{graphicx} 
\usepackage{thmtools}     
\usepackage{mathrsfs}     
\usepackage{xcolor} 
\usepackage{thm-restate}

\usepackage[capitalize,noabbrev]{cleveref}

\theoremstyle{plain}
\newtheorem{thm}{Theorem}[section]

\newtheorem{lemma}[thm]{Lemma}
\newtheorem{coll}[thm]{Corollary}
\theoremstyle{definition}

\newtheorem{obs}[thm]{Observation}

\theoremstyle{remark}

\newcommand{\precdot}{\prec\mathrel{\mkern-5mu}\mathrel{\cdot}}
\usepackage[textsize=tiny]{todonotes}

\usepackage{tikz}
\tikzstyle{L} = [rectangle, rounded corners, minimum width=2cm, minimum height=0.5cm,text centered, draw=black, fill=white!30]
\tikzstyle{T} = [rectangle, rounded corners, minimum width=2cm, minimum height=0.5cm,text centered, draw=black, fill=black!10]
\tikzstyle{C} = [rectangle, rounded corners, minimum width=2cm, minimum height=0.5cm,text centered, draw=black, fill=black!5]

\usepackage{authblk}

\title{On a Combinatorial Problem Arising in Machine Teaching}
\author[1]{Brigt Håvardstun}
\author[2]{Jan Kratochv\'{\i}l}
\author[1]{Joakim Sunde}
\author[1]{Jan Arne Telle}
\affil[1]{Department of Informatics, University Of Bergen, Bergen, Norway}
\affil[2]{Department of Applied Mathematics, Charles University, Praha, Czech Republic}

\date{\today}
\begin{document}
\maketitle

\begin{abstract}
We study a model of machine teaching where the teacher mapping is constructed from a size function on both concepts and examples. The main question in machine teaching is the minimum number of examples needed for any concept, the so-called teaching dimension. A recent paper \cite{Greedy} conjectured that the worst case for this model, as a function of the size of the concept class, occurs when the consistency matrix contains the binary 
representations of numbers from zero and up. In this paper we prove their conjecture.
       The result can be seen as a generalization of a theorem resolving the edge isoperimetry problem for hypercubes \cite{1976}, 
       and our proof is based on a lemma of  \cite{graham}.
\end{abstract}


\section{Introduction}
In formal models of \emph{machine learning} \cite{PAC} we have a concept class $C$ of
possible hypotheses, an unknown target concept $c^* \in C$
and training data given by correctly labelled random examples. The concept class $C$ is given by a binary matrix $M$ whose rows are concepts and whose column set is the domain of examples $X$, with $M(c,x)=1$ if $c$ is consistent with $(x,1)$. 
In formal models of {\em machine teaching} a set of
labelled examples $w$ called a {\em witness}
is instead carefully chosen by a teacher $T$, i.e. $T(c^*)=w$, so the learner can reconstruct $c^*$.
The common goal is to keep the {\em teaching dimension}, i.e., the cardinality of the witness
set, $\max_{c \in C}|T(c)|$, as small as possible. In recent years, the field of machine teaching has
seen various applications in fields like 
pedagogy~\cite{SHAFTO201455}, 
trustworthy AI~\cite{zhu2018overview},
reinforcement learning \cite{DBLP:conf/aaai/ZhangBMS021}, active learning \cite{DBLP:conf/nips/WangSC21} and
explainable AI~\cite{yang2021mitigating}.

Various models of machine teaching have been proposed, e.g. the classical teaching dimension model \cite{goldman1995complexity}, the optimal teacher model~\cite{balbach2008measuring}, recursive teaching~\cite{zilles2011models},  preference-based teaching~\cite{gao2017preference}, no-clash teaching~\cite{no-clash}, and probabilistic teaching \cite{DBLP:conf/ijcai/FerriHT22}. In \cite{telle2019teaching} a model focusing on teaching size is introduced, and in \cite{Greedy} an algorithm called Greedy constructing the teacher mapping in this model is given. 

Greedy assumes two total orderings $\precdot_C$ on $C$ and $\precdot_X$ on $X$, with  $\precdot_X$ extended to $\precdot_W$ on subsets of labelled examples $W = 2^{X \times \{0,1\}}$ by shortlex ordering.
In the Greedy algorithm the teacher defines its mapping iteratively: go through $W$ in the order of $\precdot_W$, and for a given witness $w=\{(x_1,b_1)...(x_q,b_q)\}$ find the earliest (in $\precdot_C$ order) $c \in C$ consistent with $w$ (i.e. with $M(c,x_i)=b_i$ for all $1 \leq i \leq q$) such that $T(c)$ is not yet defined, then set $T(c)=w$ and continue with next witness (if no such $c$ exists then drop this $w$).

To compare the teaching dimension achievable by Greedy to that of other models, the authors of \cite{Greedy} argued as follows to show that if a large witness is used then this is because $|C|$ is large:
If Greedy assigns $T(c)=w$ for some  $w=\{(x_1,b_1)...(x_q,b_q)\}$ 
then 
we may ask, why was $c$ not assigned to a smaller witness?  Assuming there are $|X|=n$ examples, then any subset $Q \subseteq X$ of size $q-1$ when labelled consistent with $c$ has already been tried by Greedy, and hence some other concept must already have been assigned to any such $Q$, and all these concepts are distinct. This means we must have taught ${n \choose q-1}=k$ other concepts already. But then we have already taught at least $k+1$ concepts and we can again ask why were any of these not taught by a smaller witness of size $q-2$? It must be that any such witnesses (labelled to be consistent with some concept among the k+1 we already have) must have been used to teach other, again distinct, concepts. 

Note that, to verify how many distinct witnesses
exist, corresponding to new concepts, that are labelled consistently with one of these $k + 1$ concepts, one must
sum up the number of distinct rows when projecting on $q-2$ columns, for all
choices of these columns. Note that the number of distinct rows, i.e witnesses and hence number of concepts, when projecting on $q-2$ columns,
for all choices of these columns, depends on the matrix $M$ you do the projection on. 
As the authors of \cite{Greedy} wanted a lower bound on number of concepts, they needed
to find the matrix $M$ minimizing the sum of unique rows after doing the projection.
Continuing to argue like this, in order to achieve a lower bound on the size of $C$ when Greedy uses a witness of size $q$, the authors of \cite{Greedy} arrive at the following combinatorial question. What is the binary matrix $M$ on $k$ distinct rows and $n$ columns that would give the smallest sum when projecting on $q$ columns? 
They conjectured that this was achieved by the matrix $H_{n,k}$ consisting of the $k$ rows corresponding to the binary representations of the numbers between zero and $k-1$, with leading 0s to give them length $n$. In this paper we prove this conjecture (note that through personal communication we were made aware of this conjecture months before its publication on arxiv).

 When $q=n-1$ this minimization question is equivalent to asking for the induced subgraph on $k$ vertices of the hypercube of dimension $n$ having the maximum number of edges, since the sum mentioned above for a matrix $M$ is then $k$ times $n$ minus the number of edges in the subgraph induced by the nodes corresponding to the rows of $M$ (since the dimension of the hypercube this edge crosses, when that column is left out, we get two identical rows corresponding to the endpoints of the edge). This has been shown 
 \cite{1976} 
 to be achieved by $H_{n,k}$. 
 The above question is called the edge isoperimetry problem for the hypercube. The maximum number of edges of an induced subgraph on $k$ vertices of the hypercube has been studied extensively in \cite{1974,1975,1976,1990,2013} to name a few articles.
 The result we give in this paper is thus a generalization of the edge isoperimetry problem on the hypercube, as we show that $H_{n,k}$ is the solution not only when $q=n-1$, but for all values of $1 \leq q \leq n$.

 The rest of our paper is organized as follows.
In Section \ref{sec:two} we give the formal definition of the conjecture. In Section \ref{sec:three} we show that the conjecture would be settled if we could prove a stronger theorem. Then in Section \ref{sec:four} we prove this stronger theorem, based on an old result from \cite{graham}.

\section{Statement of the main theorem} \label{sec:two}
Let \(M\) be a \(k \times n \) binary matrix whose all \(k\) rows are distinct. Let \(\mathcal{M}_{n,k} \) be the set of all such matrices. For any binary matrix \(A\), let \(dif(A)\) denote the number of unique rows in the matrix \(A\). For \(Q \subset \{1,2,...\} \), let \(M(Q)\) be the sub-matrix of \(M\in {\mathcal{M}_{n,k}}\) formed by taking the columns with indices from \(Q\). Finally for integers \(a\) and \(b\) where \(a \leq b \) let \([a,b] = \{a,a+1,...,b\} \). Our main interest is the number

\[m_q(M) = \sum_{Q \in {[1,n] \choose q }} dif(M(Q)). \]

For fixed positive integers \(k,n\) and \(q\), we are interested in finding a matrix \(M \in \mathcal{M}_{k,n} \) with the minimum value of \(m_q(M) \). Let \(m_q(n,k)\) be this minimum value, i.e.,
\[m_q(n,k) = \min_{M \in \mathcal{M}_{n,k}} m_q(M).  \]


We show that the $k \times n$ binary matrix $H_{n,k}$ whose rows are the binary representations of all numbers between zero and $k-1$ achieves this minimum value of \(m_q(n,k)\). It will be useful for us to use the following recursive definition of $H_{n,k}$. Let \(\mathbf{0} \) be the all 0 row vector and let \(\mathbf{0}^T\) be the all 0 column vector, and similarly for \(\mathbf{1}\) and \(\mathbf{1}^T\). Then 

\[H_{n,k} = \begin{cases}
    \mathbf{0} & k=1 \\
    \begin{pmatrix}
    H_{n-1, \lceil \frac{k}{2} \rceil } & \mathbf{0}^T\\
    H_{n-1, \lfloor \frac{k}{2} \rfloor } & \mathbf{1}^T 
\end{pmatrix} & k>1 \\
    
\end{cases}  \]


Let \(h_q(n,k) = m_q(H_{n,k})  \). Our goal is thus to prove the following theorem.
\begin{thm}
    \label{main_theorem}
    For any positive integers \(q,n,k\) where \(q\leq n \) and \(k \leq 2^n \),
    \[m_q(n,k) = h_q(n,k). \] 
\end{thm}

Here is a diagram showing how we will prove Theorem \ref{main_theorem}.

\begin{center}
    \resizebox{\linewidth}{!}{
    \begin{tikzpicture}[]
    \node [L] (L3_2) at (0,0) {Lemma 3.2};
    \node [L] (L3_3) at (0, -2) {Lemma 3.3};
    \node [L] (L3_4) at (3,-1) {Lemma 3.4};

    \draw[->] (L3_2) -- (L3_4);
    \draw[->] (L3_3) -- (L3_4);

    \node [L] (L4_1) at (0,-4) {Lemma 4.1};
    \node [C] (C4_2) at (3,-4) {Cor. 4.2};
    \draw[->] (L4_1) -- (C4_2);

    \node [L] (L4_3) at (0,-6) {Lemma 4.3 \cite{graham}};
    \node [L] (L4_5) at (3.5,-6) {Lemma 4.5};
    \node [L] (L4_6) at (6,-6) {Lemma 4.7};
    
    \draw[->] (L4_3) -- (L4_5);
    \draw[->] (L4_5) -- (L4_6);

    \node [T] (T3_1) at (9, -5) {Thm. 3.1};
    \draw[->] (L4_6) -- (T3_1);
    \draw[->] (C4_2) -- (T3_1);

    \node [T] (T1) at (12,-3) {Thm. 2.1};
    \draw[->] (L3_4) -- (T1);
    \draw[->] (T3_1) -- (T1);



    \end{tikzpicture}
    }
\end{center}

\section{A sufficient condition}
\label{sec:three}


The goal of this section is to prove that the following 
theorem (whose proof we leave to the next section)  implies Theorem \ref{main_theorem}.
\begin{restatable}[]{thm}{toprove}
    \label{conjecture_A} For any positive integers \(q,n,k\) where \(q\leq n \) and \(k \leq 2^n \),
    \[\min_{\lceil \frac{k}{2} \rceil \leq x \leq k-1 } h_q(n,x) + h_{q-1}(n-1,k-x) \]
    \[= h_q(n, \lceil \frac{k}{2} \rceil ) + h_{q-1}(n-1, \lfloor \frac{k}{2} \rfloor ). \]
\end{restatable}

\noindent
which is just stating that the minimum value of the expression on the left occurs when \(x =\lceil \frac{k}{2} \rceil \).

\begin{lemma}
\label{h_reccurence}
    The $h$ numbers satisfy the recurrence relation
$$h_q(n,1)={n\choose q}$$
and
$$h_q(n,k)=h_q(n,\lceil\frac{k}{2}\rceil) + 
h_{q-1}(n-1,\lfloor\frac{k}{2}\rfloor)$$
for $k>1$.
\end{lemma}

\begin{proof}
Let $Q$ be a $q$-element subset of the column-index set $\{1,2,\ldots,n\}$. If $n\not\in Q$, then each  of the bottom $\lfloor\frac{k}{2}\rfloor$ rows of $H_{n,k}(Q)$ appears as a row among the  $\lceil\frac{k}{2}\rceil$ top ones, and hence $m_q(H_{n,k}(\{1,2,\ldots,n-1\}))=m_q(H_{n-1,\lceil\frac{k}{2}\rceil})=h_q(n-1,\lceil\frac{k}{2}\rceil)$.
If $n\in Q$, every row from the bottom $\lfloor\frac{k}{2}\rfloor$ rows of $H_{n,k}(Q)$ differs from any row from the $\lceil\frac{k}{2}\rceil$ top ones, and so the sum over those $Q$ that contain $n$ contributes exactly $m_{q-1}(H_{n-1,\lceil\frac{k}{2}\rceil})+m_{q-1}(H_{n-1,\lfloor\frac{k}{2}\rfloor})=h_{q-1}(n-1,\lceil\frac{k}{2}\rceil) +
h_{q-1}(n-1,\lfloor\frac{k}{2}\rfloor)$. Thus 
$$h_q(n,k)$$
$$=h_q(n-1,\lceil\frac{k}{2}\rceil)+
h_{q-1}(n-1,\lceil\frac{k}{2}\rceil) +
h_{q-1}(n-1,\lfloor\frac{k}{2}\rfloor).$$ This can be slightly simplified as follows. Note that $h_q(n-1,\lceil\frac{k}{2}\rceil)+
h_{q-1}(n-1,\lceil\frac{k}{2}\rceil)$ is exactly the contribution of the $\lceil\frac{k}{2}\rceil$ top rows of $H_{n,k}$ to $m_q(H_{n,k})$, i.e., $m_q((H_{n-1,\lceil\frac{k}{2}\rceil} \mathbf{0}^T))$ what equals $m_q((\mathbf{0}^T H_{n-1,\lceil\frac{k}{2}\rceil} ))=m_q(H_{n,\lceil\frac{k}{2}\rceil})=h_q(n,\lceil\frac{k}{2}\rceil)$ and the claim follows. 
\end{proof} 

\begin{lemma}
    \label{m_reccurence} For any positive integers \(q,n,k\) where \(q\leq n \) and \(k \leq 2^n \),
\[m_q(n,k) \geq \min_{\lceil \frac{k}{2} \rceil \leq x \leq k-1 } m_q(n,x) + m_{q-1}(n-1,k-x) \]
\end{lemma}

\begin{proof}
Let $A\in {\cal M}_{n,k}$ be a matrix that minimizes $m_q$ over ${\cal M}_{n,k}$, i.e., it satisfies $m_q(A)=m_q(n,k)$.

If $k=1$, then every $Q\in {[1,n]\choose q}$ contributes 1 to the sum $\sum_Q \mbox{dif}(M(Q))$, and hence $m_q(A)={n\choose q}$.

Let $k>1$. Suppose w.l.o.g. that the last column contains both 0's and 1's. Let $y$ be the number of 0's in it, and assume that the 0's are in rows $1,\ldots,y$ and the 1's in rows $y+1,\ldots,k$, with $y\ge k-y$, i.e., $y\ge \lceil\frac{k}{2}\rceil$. Let $T$ be the submatrix of $A$ determined by rows $1,\ldots, y$ and columns $1\,ldots, n-1$, and let $B$ be the submatrix determined by rows $y+1,\ldots, k$ and columns  $1,\ldots, n-1$, i.e., 

$$
A=\left( 
\begin{array}{cc}
T & \mathbf{0}^T\\
B & \mathbf{1}^T
\end{array}
\right).
$$
We further denote by $T^*=(T\; \mathbf{0}^T)$ the sub matrix of $A$ formed by its top $\lceil\frac{k}{2}\rceil$ rows.

Let $Q$ be a $q$-element subset of the column-index set $\{1,2,\ldots,n\}$. If $n\not\in Q$, then $\mbox{dif}(A(Q))\ge \mbox{dif}(T(Q))$. If $n\in Q$, 
$\mbox{dif}(A(Q))=\mbox{dif}(T(Q-\{n\}))+
\mbox{dif}(B(Q-\{n\}))=\mbox{dif}(T^*(Q))+
\mbox{dif}(B(Q-\{n\}))$. Therefore
$$
m_q(A)=\sum_{Q:n\not\in Q}\mbox{dif}(A(Q))
+\sum_{Q:n\in Q}\mbox{dif}(A(Q))\ge 
$$
$$\ge 
\sum_{Q\in {\{1\ldots n-1\}\choose q}}\mbox{dif}(T(Q))
$$ 
$$+ \sum_{Q'\in {\{1\ldots n-1\}\choose q-1}}\mbox{dif}(T(Q'))
+ \sum_{Q'\in {\{1\ldots n-1\}\choose q-1}}\mbox{dif}(B(Q'))= 
$$

$$=
\sum_{Q\in {\{1\ldots n-1\}\choose q}}\mbox{dif}(T(Q))
+\sum_{Q\in {\{1\ldots n\}\choose q}, n\in Q}\mbox{dif}(T^*(Q))
$$
$$+\sum_{Q'\in {\{1\ldots n-1\}\choose q-1}}\mbox{dif}(B(Q'))=
$$
$$ = m_q(T^*)+m_{q-1}(B)
\ge
$$
$$\geq 
m_q(n,y)+ m_{q-1}(n-1,k-y)\geq
$$
$$
\geq 
\min_{\lceil\frac{k}{2}\rceil\le x\le k-1}m_q(n,x)+m_{q-1}(n-1,k-x).
$$

\end{proof}

\begin{lemma}
    \label{equivalent}
    Theorem \ref{conjecture_A} implies Theorem \ref{main_theorem}
 \end{lemma}

\begin{proof} 
Certainly $m_q(n,k)\le h_q(n,k)$, we prove the other inequality by induction on $k$. The base case $k=1$ follows from $m_q(n,1)=h_q(n,1)={n\choose q}$.

Suppose $k>1$. Lemmas \ref{h_reccurence} and \ref{m_reccurence} imply that
$$m_q(n,k) \ge 
\min_{\lceil\frac{k}{2}\rceil\le x\le k-1}m_q(n,x)+m_{q-1}(n-1,k-x) \ge
$$
(by the induction hypothesis)
$$
\ge
\min_{\lceil\frac{k}{2}\rceil\le x\le k-1}h_q(n,x)+h_{q-1}(n-1,k-x) =
$$
(by Theorem \ref{conjecture_A})
$$
= h_q(n,\lceil\frac{k}{2}\rceil)+h_{q-1}(n-1,\lfloor\frac{k}{2}\rfloor) =
$$
(by Lemma~\ref{h_reccurence})
$$= h_q(n,k).
$$

\end{proof}

\section{Proving Theorem \ref{conjecture_A}}
\label{sec:four}

In this section we will prove Theorem~\ref{conjecture_A},  and show that \(h_q(n,x)\) "increases" at least as fast as \(h_{q-1}(n-1,k-x)\) "decreases" when \(x\) starts at \(\lceil \frac{k}{2} \rceil \) and increases until \(k-1\). If we can show this, then the minimum value of  \( h_q(n,x) + h_{q-1}(n-1,k-x)\) will be at \(x = \lceil \frac{k}{2} \rceil \). To be more precise, we want to show that 
\begin{equation}
\label{DecByJThe3.1}
\centering
\begin{split}
h_q(n,\lceil \frac{k}{2}\rceil +j) &- h_q(n,\lceil \frac{k}{2} \rceil ) \\
\geq h_{q-1}(n-1, \lfloor \frac{k}{2} \rfloor )& - h_{q-1}(n-1,\lfloor \frac{k}{2} \rfloor - j )
\end{split}
\end{equation}
for any \(j\ge 1  \) such that \(\lceil \frac{k}{2} \rceil + j \leq k-1\). 

We first need to understand the behavior of the \(h_q(n,k) \) numbers as \(k\) increases or decreases. Let $|x|$ denote the Hamming weight (number of 1's) in the binary representation of integer $x$. We recall that the binomial coefficient $n\choose k$ by definition evaluates to 0 when $k<0$ or $k>n$. Similarly, we define the 
boundary values of $h_q(n,k)$ for $q=0$ and $k=0$ as $h_0(n,k)=1$ (for $k>0$) and $h_q(n,0)=0$. 

\begin{lemma}
    \label{one_bigger}
    For any integers \(x,q,n\) such that \(0\le x \leq 2^n-1 \) and \(0\le q\leq n \), we have
    \[h_q(n,x+1) = h_q(n,x) + {n-|x| \choose q-|x| }, \]
    and for integers \(x,q,n \) such that \(1 \leq x \leq 2^{n-1} \) and \( 1 \leq q \leq n \), we have
    \[h_{q-1}(n-1,x-1) = h_{q-1}(n-1,x) - {n-1-|x-1| \choose q-1-|x-1| }. \]
\end{lemma}

\begin{proof}
    We prove the first formula, the second one then follows directly by applying the first one for $x-1, q-1$ and $n-1$. \\
    For the 
    boundary values of $q$ and $x$, we have $h_0(n,1)=1=0+1=h_0(n,0)+{n\choose 0}$, $h_0(n,x+1)=1=1+0=h_0(n,x)+{n-|x|\choose -|x|}$ for $x\ge 1$, and $h_q(n,1)={n\choose q}=0+{n\choose q}=h_q(n,0)+{n-|0|\choose q-|0|}$.\\
    For the notrivial cases, suppose that $q\ge 1$ and $x\ge 1$.
    The only difference between \(H_{n,x} \) and \(H_{n,x+1} \) is that \(H_{n,x+1}\) has one extra row, which is the binary representation of \(x\) with zeroes padded to the left if needed. Let \(S\) be the set of column indices where the last row of \(H_{n,x+1} \) has a \(1\). \\
    We first observe that \(dif(H_{n,x}(Q)) = dif(H_{n,x+1}(Q)) \) whenever \(S \not\subseteq Q \subseteq \{1,2,\ldots,n\}\). To see this let \(i \in S\backslash Q \) and $y$ be the number with binary representation having the same entry as $x$ in the positions belonging to $Q$ and 0's in all other positions. Then $y<x$ and \(H_{n,x}\) contains a row which is the binary representation of $y$. Since this row of \(H_{n,x} \) is equal to the last row of \(H_{n,x+1}\) when only looking at the columns with indices in \(Q\), \(dif(H_{n,x}(Q)) = dif(H_{n,x+1}(Q)) \). \\
    Then we see that \(dif(H_{n,x+1}(Q) = dif(H_{n,x}(Q)) +1  \) whenever \(S \subseteq Q \). This is because there is no row in \(H_{n,x} \) where all the columns with indices in \(S\) are equal to \(1\), since the number of this row would be greater or equal to \(x\). \\
    So we are left with counting how many subsets \(Q\) of \(\{1,...,n\} \) satisfy \(S \subseteq Q \) and \(|Q| = q  \). This is exactly \({n-|S| \choose q-|S|  } = {n-|x| \choose q-|x| } \). 

\end{proof}

\begin{coll}
    \label{plus_more}
    For any integers \(q,n,x,j\) such that $0\le q\le n$, $0\le x$,  \(1\leq j \)  and \(x+j \leq 2^n \), we have
    \[h_q(n,x+j) = h_q(n,x) + \sum_{i=x}^{x+j-1}{n-|i| \choose q-|i| }.  \]   
    Moreover, whenever \(1 \leq q \leq n\) and \(1 \leq j\le x \leq 2^{n-1} \),  we have
               \[h_{q-1}(n-1,x-j) = h_{q-1}(n-1,x) - \sum_{i=x-j}^{x-1}   
 {n - 1 - |i| \choose q-1-|i| },\]   
 and whenever \(1 \leq q \leq n \) and \(1 \leq j\le x-1\leq 2^{n-1} \), we have
 \[h_{q-1}(n-1,x-j-1)\]\[ = h_{q-1}(n-1,x-1) - \sum_{i=x-j-1}^{x-2}   
 {n - 1 - |i| \choose q-1-|i| }.\] 
\end{coll}
\begin{proof}
     The first two formulae follow from Lemma~\ref{one_bigger} by induction on $j$, the third formula follows from the second by substituting $x-1$ for $x$.
\end{proof}

In view of this corollary, the inequality (\ref{DecByJThe3.1}) is equivalent to the claim that our goal is to prove that
$$ \sum_{i=\lceil\frac{k}{2}\rceil}^{\lceil\frac{k}{2}\rceil+j-1} {n-|i|\choose q-|i|} \ge 
\sum_{i=\lfloor\frac{k}{2}\rfloor-j}^{\lfloor\frac{k}{2}\rfloor-1} {n-1-|i|\choose q-1-|i|}$$
holds true for all feasible $q,n,k$ and $j$.

We first show some useful properties of Hamming weights
which extend the following lemma from \cite{graham} whose proof was finalized in \cite{missing_case}.

\begin{lemma}(\cite{graham,missing_case})
    \label{graham}
    Let \(s,t\) be non-negative integers. Then there exists a bijective mapping \(\theta : [0,r] \rightarrow [s,s+r] \) such that \(|\theta(k)| \geq |k| \) for every $k\in [0,r]$.
\end{lemma}

We will need a generalization of this lemma whose proof depends on the following observation:

\begin{obs}
        \label{reverse_triangle}
        Let  \(x \ge t\) be non-negative integers. Then \(|x-t| \geq |x|-|t|  \).
\end{obs}

\begin{proof}
    This follows directly from the standard subtraction algorithm for integers in binary representation.
\end{proof}

\begin{lemma}
    \label{number_theory}
    Let \(s,r,t\) be non-negative integers such that \(r,t\geq 1 \) and \(s\geq r+t-1 \). Denote by \(T=[s,s+r-1]\) and \(B =  [s-r-t+1,s-t]  \). Then there exists a bijective mapping \(\theta : T \rightarrow B  \) such that \(|\theta(x)| \geq |x|-|t|  \) for all \(x \in T \).
\end{lemma}
\begin{proof}

    Our proof works by induction on \(r\). When \(r=1\), we have \(T=\{s\} \) and \(B = \{s-t\} \). The only possible mapping \(\theta \) then simply maps \(s\) to \(s-t \) and we see that \(|\theta(s)| = |s-t| \geq |s|-|t| \) by Observation~\ref{reverse_triangle}. Thus, the base case \(r=1\) is established for all values of \(t \geq 1\).
    
    Let \(r>1 \).
    Create two matrices with \(r\) rows each $$M_T = \begin{pmatrix}
        \overrightarrow{s+r-1} \\
        \vdots \\
        \overrightarrow{s+1} \\
        \overrightarrow{s} \\
    \end{pmatrix} \mbox{ and } M_B = \begin{pmatrix}
        \overrightarrow{s-t} \\
        \overrightarrow{s-t-1} \\
        \vdots \\
        \overrightarrow{s-r-t+1} \\
    \end{pmatrix} $$ where \(\overrightarrow{x}\) is the base 2 representation of \(x\) as a binary vector with \(0\)-s padded to the left so that all  vectors have the same length. Finally let \(M = \begin{pmatrix}
        M_T \\
        M_B
    \end{pmatrix} \).
    
    Reformulating the lemma in this matrix context we seek a bijective mapping \(\theta\) of the rows of $M_T$ to the rows of $M_B$ such that \(|\theta(x)| \geq |x|-|t| \) holds true for every row $x$ of $M_T$. (With a slight abuse of notation we write \( \theta: M_T \rightarrow M_B  \).)  The induction hypothesis states that this holds true, for this value of $t$, if the number of rows of each matrix is less than $r$. \\
    Without loss of generality we may assume that the first (leftmost) column of \(M\) contains at least one 0 and at least one 1 (since we could disregard this column otherwise).  
    Then if we look at the first column of \(M\), there will be a point where a \(1\) appears for the first time, when moving through the rows from the bottom row up. This could happen either in the \(M_T\) part or in the \(M_B\) part of the matrix. We will deal with these 2 cases separately. \\ 

    \textbf{Case 1} ({The first leftmost \(1\) appears in the \(M_T\) part of the matrix})
         We divide both the \(M_T\) and \(M_B\) matrices further and write $M$ as \[M = \begin{pmatrix}
             T_1 \\
             T_2 \\
             B_2 \\
             B_1
         \end{pmatrix} \]
    where the bottom row of \(T_1\) is the row where the first \(1\) appears (thus that row is \(100...0\)), with $T_2$ being the remainder of $M_T$, and we let \(B_1\) have the same number of rows as \(T_1\). We will map $T_1$ to $B_1$ and $T_2$ to $B_2$.
    Since $T_2$ and $B_2$ have fewer rows than $r$, it follows by the induction hypothesis that, for the same value of $t$, there exists the required mapping \(\theta_1 : T_2 \rightarrow B_2 \). Note also that this is vacuously true if $T_2$ and $B_2$ are empty. Now if we ignore the first column of \(T_1\), then \(T_1(\{2,3,...\})\) is the binary representation of the numbers \(0,1,...,|T_1|-1\). So by Lemma \ref{graham} there is a mapping \(\theta_2 : T_1(\{2,3,...\}) \rightarrow B_1 \) such that \(|\theta_2(x)| \geq |x| \) for every $x$. Adding back the first column of \(T_1\) and using the same mapping between the rows as $\theta_2$, we get a mapping \(\theta_3 : T_1 \rightarrow B_1 \) where \(|\theta_3(x)| \geq |x|-1 \) for every $x$ (since the Hamming weight of \(x\) increases by 1). Clearly \(|x| - 1 \geq |x|-|t| \) when \(t \geq 1 \), hence combining \(\theta_1 \) and \(\theta_3 \) gives us a bijective mapping \(\theta : T \rightarrow B \) with the required properties. \\

    \textbf{Case 2} (The first leftmost 1 appears in the \(M_B\) part)
    We divide the matrix in a similar way as in the first case \[M = \begin{pmatrix}
        T_1 \\
        T_2 \\
        B_2 \\
        B_1
        
    \end{pmatrix} \]
    so that the bottom row of \(B_2\) is the row where the first \(1\) appears in the leftmost column (so this row is \(100...0\)) and  we let \(T_2\) have the same number of rows as \(B_2\). For a binary vector \(x\) let \(\overline{x} \) be the complement of \(x\), so \(\overline{x} = (1,1,1,...,1) - x \). For a binary matrix $A$, let \(\overline{A} \) be the matrix whose rows are the complements of the rows of \(A\).

    By the induction hypothesis, as there are fewer rows and we have the same value of $t$, there is a mapping \(\theta_1 : T_2 \rightarrow B_2 \) with the required properties. 
    The top row of \(B_1\) is \((011...1)\) so if we look at \(\overline{B_1} \), the top row will be  \((100...0)\), the next row will be \((100...01) \) and so on, meaning that if we ignore the first column we are counting up from \(0\) in binary. Lemma \ref{graham} then gives us a mapping  \(\theta_2 : \overline{B_1(\{2,3,...\})} \rightarrow \overline{T_1}  \) with \(|\theta_2(x)| \geq |x| \) for every $x$. Adding back the first column but keeping the row mapping we get a mapping \(\theta_3 : \overline{B_1} \rightarrow \overline{T_1} \) where \(|\theta_3(x)| \geq |x|-1 \). 
    Now define \(\theta_4 : B_1 \rightarrow T_1 \) by \(\theta_4(x) = \overline{\theta_3(\overline{x})} \) and let \(||x||\) be the length of the vector \(x\). 
    \begin{align}
        |\theta_4(x)| - |x| &= |\overline{\theta_3(\overline{x}) } | - |x| \\
        &= ||x|| - |\theta_3(\overline{x})| - (||x||-|\overline{x}|) \\
        &= |\overline{x}| -|\theta_3(\overline{x})|\\
        &\leq 1
    \end{align}
    To get (5) we use the fact that \(|\theta_3(x)| \geq |x|-1 \).
    
   We will now look at the inverse mapping \(\theta_4^{-1} : T_1 \rightarrow B_1 \). For any \( y \in T_1 \), there is an \(x \in B_1 \) such that \(\theta_4(x) = y \). We just showed that \(|\theta_4(x)|-|x| \leq 1 \) which is the same as \(|y|-|x| \leq 1 \) which we can rewrite as \(|y|-|\theta_4^{-1}(y)| \leq 1 \). Multiplying both sides by \(-1\) we get \(|\theta_4^{-1}(y)| \geq |y| -1 \). Combining \( \theta_4^{-1} \) and \(\theta_1 \) in the natural way we get the desired mapping \(\theta : T \rightarrow B \) satisfying \(|\theta(x)| \geq |x|-1 \geq |x|-|t| \) for every $x$. 
   
   Thus the lemma is proven for any \(r,t \geq 1 \) and any $s \geq r+t-1$.  
 \end{proof}

We are now set to prove that the sum in the first formula of the Corollary~\ref{plus_more} is always larger than the sums in the second and third formulae. 
We will need the following well known observation: 
\begin{obs}
    For any integers \(n,k,j\) such that \(0 \leq j\leq k \leq n\), 
    \[{n \choose k} \geq {n-j \choose k-j}.\]
    \label{PascalTriangle}
\end{obs}

\begin{lemma}
\label{finally}    
For all positive integers $q,n,x,j$ such that $x+j-1\le 2^n$ and $x-j\ge 0$,
\[\sum_{i=x}^{x+j-1}{n-|i| \choose q-|i| } \geq \sum_{i=x-j}^{x-1}   
 {n - 1 - |i| \choose q-1-|i| },\]
 and if $x-j\ge 1$, then
\[\sum_{i=x}^{x+j-1}{n-|i| \choose q-|i| } \geq \sum_{i=x-j-1}^{x-2}   
 {n - 1 - |i| \choose q-1-|i| }.\]
 
\end{lemma}

\begin{proof}
    We show that there exists a bijection \(\theta \) from \([x,x+j-1] \) to \([x-j,x-1] \)  (or to \([x-j-1,x-2] \), respectively) such that \(|\theta(i)| \geq |i|-1 \) for all $i$. This will prove the lemma since then for every term \({n - |i| \choose q-|i| } \) in the sum of the left hand side there will be a corresponding term in the sum on the right side \[ {n - 1- |\theta(i)| \choose q-1-|\theta(i)| }  \]
    and we see that by Observation~\ref{PascalTriangle}
    \[ {n - 1- |\theta(i)| \choose q-1-|\theta(i)| } \leq {n-1-(|i|-1) \choose q-1-(|i|-1) } = {n-|i| \choose q-|i| }. \]
    So if such a bijection exists, then for every term in the sum on the left hand side there will be a unique element in the sum on the right hand side which is no greater than the element on the left hand side. So then the sum on the left hand side must be greater than or equal to the sum on the right hand one. Now we just need to show that there exist such bijections \(\theta \). We will use Lemma~\ref{number_theory}. \\
{\bf Case 1.} Let \(s = x, r= j \) and \(t=1\). Then by Lemma~\ref{number_theory} there is a mapping \(\theta : [x,x+j-1] \rightarrow [x-1,x-j] \}  \) such that \(|\theta(i)| \geq |i|-|t| = |i|-1 \) for every $i$, which is the first bijection we wanted. \\
{\bf Case 2.} If we set \(t=2 \), then Lemma~\ref{number_theory} gives us a mapping \(\theta : [x,x+j-1] \rightarrow [x-2,x-j-1]  \) such that \(|\theta(i)| \geq |i|-|t| = |i|-1 \) for every $i$, which is the second bijection we wanted. 
\end{proof}

We are now ready to prove Theorem \ref{conjecture_A}. 

\toprove*
\begin{proof}
    We consider the two cases depending on whether \(k \) is either even or odd. 

    \noindent{\bf Case 1 - 
    \(k\) is even.} Set \(x = \frac{k}{2} \) and rewrite the left hand side of Theorem~\ref{conjecture_A} as
    \[\min_{0 \leq j \leq x-1 } h_q(n,x+j) + h_{q-1}(n-1,x-j ) \]
     Then by the first and second part of Corollary~\ref{plus_more} whenever \( 1 \leq j \leq x-1 \), we can rewrite the expression which is minimized as
     \[ h_q(n,x) + \sum_{i=x}^{x+j-1}{n-|i| \choose q-|i| } + h_{q-1}(n-1,x)\]
     \[ - \sum_{i=x-j}^{x-1}   
 {n - 1 - |i| \choose q-1-|i| } \]
 By Lemma~\ref{finally} we have \(\sum_{i=x}^{x+j-1}{n-|i| \choose q-|i| } \geq \sum_{i=k-j}^{x-1}   
 {n - 1 - |i| \choose q-1-|i| }\) for any \( 1\leq j \leq x-1 \), which means that the smallest value occurs when \(j=0\). \\
{\bf Case 2 - \(k\) is odd.} Set \(x=\lceil \frac{k}{2} \rceil  \) and rewrite the expression of the left hand side of Theorem~\ref{conjecture_A} as
\[\min_{0 \leq j \leq x-1 } h_q(n,x+j) + h_{q-1}(n-1,x-1-j ). \]
By the first and third part of Corollary~\ref{plus_more} we can rewrite the part which is minimized for \( 1\leq j \leq x-1 \) as
\[h_q(n,x) + \sum_{i=x}^{x+j-1}{n-|i| \choose q-|i| } + h_{q-1}(n-1,x-1)\]
\[- \sum_{i=x-j-1}^{x-2}   
 {n - 1 - |i| \choose q-1-|i| }. \]
 By Lemma~\ref{finally} we have \(\sum_{i=x}^{x+j-1}{n-|i| \choose q-|i| } \geq \sum_{i=k-j-1}^{x-2}   
 {n - 1 - |i| \choose q-1-|i| } \) for \(1 \leq j \leq x-1 \) so the smallest value of the expression will occur for \(j=0\).
 
\end{proof}

By Lemma~\ref{equivalent} this  proves Theorem~\ref{main_theorem}.

\section{Conclusion}

We have proven a conjecture of \cite{Greedy}, that will allow to compare the teaching dimension achievable by Greedy to that of other machine teaching models. 

For a general consistency matrix $M$, the performance of Greedy can be analyzed by focusing on the $m_q(M)$ value, the sum of the number of distinct rows over all submatrices on $q$ columns. Let us consider the complexity of computing $m_q(M)$ given as input the binary matrix $M$ on $k$ rows and $n$ columns. There is a straightforward algorithm with runtime  $O(n^q k q \log k)$. The question arises if computing $m_q(M)$ is FPT (Fixed Parameter Tractable, see \cite{fpt}) when parameterized by $q$. In other words, is there an algorithm whose runtime is polynomial in the size of $M$, with any exponential dependency restricted to $q$ only? We leave this as an open problem.

\section*{Acknowledgements}
This project is supported by NRC project (329745) Machine Teaching for Explainable AI.

\bibliography{main}

\begin{thebibliography}{10}

\bibitem{2013}
Geir Agnarsson.
\newblock On the number of hypercubic bipartitions of an integer.
\newblock {\em Discrete Mathematics}, 313(24):2857--2864, 2013.

\bibitem{balbach2008measuring}
Frank~J Balbach.
\newblock Measuring teachability using variants of the teaching dimension.
\newblock {\em Theoretical Computer Science}, 397(1-3):94--113, 2008.

\bibitem{fpt}
Marek Cygan, Fedor~V. Fomin, Lukasz Kowalik, Daniel Lokshtanov, D{\'{a}}niel Marx, Marcin Pilipczuk, Michal Pilipczuk, and Saket Saurabh.
\newblock {\em Parameterized Algorithms}.
\newblock Springer, 2015.

\bibitem{1975}
Hubert Delange.
\newblock Sur la fonction sommatoire de la fonction" somme des chiffres".
\newblock {\em Enseign. Math.}, 21(2):31--47, 1975.

\bibitem{no-clash}
Shaun Fallat, David Kirkpatrick, Hans~U Simon, Abolghasem Soltani, and Sandra Zilles.
\newblock On batch teaching without collusion.
\newblock {\em Journal of Machine Learning Research}, 24:1--33, 2023.

\bibitem{DBLP:conf/ijcai/FerriHT22}
C{\`{e}}sar Ferri, Jos{\'{e}} Hern{\'{a}}ndez{-}Orallo, and Jan~Arne Telle.
\newblock Non-cheating teaching revisited: {A} new probabilistic machine teaching model.
\newblock In Luc~De Raedt, editor, {\em Proceedings of the Thirty-First International Joint Conference on Artificial Intelligence, {IJCAI} 2022, Vienna, Austria, 23-29 July 2022}, pages 2973--2979. ijcai.org, 2022.

\bibitem{Greedy}
Cèsar Ferri, Dario Garigliotti, Brigt Arve~Toppe Håvardstun, Josè Hernández-Orallo, and Jan~Arne Telle.
\newblock When redundancy matters: Machine teaching of representations, 2024.
\newblock arXiv:2401.12711.

\bibitem{gao2017preference}
Ziyuan Gao, Christoph Ries, Hans~Ulrich Simon, and Sandra Zilles.
\newblock Preference-based teaching.
\newblock {\em Journal of Machine Learning Research}, 18:31:1--31:32, 2017.

\bibitem{goldman1995complexity}
Sally~A Goldman and Michael~J Kearns.
\newblock On the complexity of teaching.
\newblock {\em Journal of Computer and System Sciences}, 50(1):20--31, 1995.

\bibitem{graham}
Ron~L Graham.
\newblock On primitive graphs and optimal vertex assignments.
\newblock {\em Annals of the New York academy of sciences}, 175(1):170--186, 1970.

\bibitem{1990}
Daniel~H Greene and Donald~Ervin Knuth.
\newblock {\em Mathematics for the Analysis of Algorithms}, volume 504.
\newblock Springer, 1990.

\bibitem{1976}
Sergiu Hart.
\newblock A note on the edges of the n-cube.
\newblock {\em Discrete Mathematics}, 14(2):157--163, 1976.

\bibitem{missing_case}
Julie~C Jones and Bruce~F Torrence.
\newblock The case of the missing case: the completion of a proof by rl graham.
\newblock {\em Pi Mu Epsilon Journal}, 10(10):772--778, 1999.

\bibitem{1974}
M.~Douglas McIlroy.
\newblock The number of 1’s in binary integers: bounds and extremal properties.
\newblock {\em SIAM Journal on Computing}, 3(4):255--261, 1974.

\bibitem{SHAFTO201455}
Patrick Shafto, Noah~D. Goodman, and Thomas~L. Griffiths.
\newblock A rational account of pedagogical reasoning: Teaching by, and learning from, examples.
\newblock {\em Cognitive Psychology}, 71:55 -- 89, 2014.

\bibitem{telle2019teaching}
Jan~Arne Telle, Jos{\'e} Hern{\'a}ndez-Orallo, and C{\`e}sar Ferri.
\newblock The teaching size: computable teachers and learners for universal languages.
\newblock {\em Machine Learning}, 108(8-9):1653--1675, 2019.

\bibitem{PAC}
Leslie~G. Valiant.
\newblock A theory of the learnable.
\newblock {\em Commun. {ACM}}, 27(11):1134--1142, 1984.

\bibitem{DBLP:conf/nips/WangSC21}
Chaoqi Wang, Adish Singla, and Yuxin Chen.
\newblock Teaching an active learner with contrastive examples.
\newblock In Marc'Aurelio Ranzato, Alina Beygelzimer, Yann~N. Dauphin, Percy Liang, and Jennifer~Wortman Vaughan, editors, {\em Advances in Neural Information Processing Systems 34: Annual Conference on Neural Information Processing Systems 2021, NeurIPS 2021, December 6-14, 2021, virtual}, pages 17968--17980, 2021.

\bibitem{yang2021mitigating}
Scott Cheng-Hsin Yang, Wai~Keen Vong, Ravi~B Sojitra, Tomas Folke, and Patrick Shafto.
\newblock Mitigating belief projection in explainable artificial intelligence via bayesian teaching.
\newblock {\em Scientific reports}, 11(1):1--17, 2021.

\bibitem{DBLP:conf/aaai/ZhangBMS021}
Xuezhou Zhang, Shubham~Kumar Bharti, Yuzhe Ma, Adish Singla, and Xiaojin Zhu.
\newblock The sample complexity of teaching by reinforcement on q-learning.
\newblock In {\em Thirty-Fifth {AAAI} Conference on Artificial Intelligence, {AAAI} 2021, Thirty-Third Conference on Innovative Applications of Artificial Intelligence, {IAAI} 2021, The Eleventh Symposium on Educational Advances in Artificial Intelligence, {EAAI} 2021, Virtual Event, February 2-9, 2021}, pages 10939--10947. {AAAI} Press, 2021.

\bibitem{zhu2018overview}
Xiaojin Zhu, Adish Singla, Sandra Zilles, and Anna~N Rafferty.
\newblock An overview of machine teaching.
\newblock {\em arXiv preprint arXiv:1801.05927}, 2018.

\bibitem{zilles2011models}
Sandra Zilles, Steffen Lange, Robert Holte, and Martin Zinkevich.
\newblock Models of cooperative teaching and learning.
\newblock {\em Journal of Machine Learning Research}, 12(Feb):349--384, 2011.

\end{thebibliography}

\end{document}